\documentclass[12pt]{article}
\usepackage{mathrsfs}
\usepackage{amsthm}
\usepackage{amssymb}
\usepackage{latexsym}
\usepackage{amsmath,amsfonts}
\usepackage{mathrsfs}
\usepackage{cases}
\usepackage{latexsym,bm}
\usepackage{indentfirst}
\usepackage{xcolor}
\usepackage{ifpdf}
\usepackage{graphicx}
\usepackage{epstopdf}
\usepackage{epsfig}
\usepackage{psfrag}
\usepackage{enumitem}
\usepackage{epstopdf}
\usepackage{verbatim}
\usepackage{color}

\usepackage[
pdfauthor={Lu},
pdftitle={matching covered},
pdfstartview=XYZ,
bookmarks=true,
colorlinks=true,
linkcolor=blue,
urlcolor=blue,
citecolor=blue,
bookmarks=true,
linktocpage=true,
hyperindex=true
]{hyperref}

\title{Wheel-like bricks and minimal matching covered graphs
\footnote{The research
 is partially supported by NSFC (No. 12271235).
\newline E-mail addresses: xiaolinghe99@163.com (X. He), flianglu@163.com (F. Lu), xuejinxin00@163.com (J. Xue) } }
\author{ Xiaoling He$^1$\footnote{Current address: School of Mathematics and Statistics, Lanzhou University, Lanzhou, China.}, Fuliang Lu$^1$, Jinxin Xue$^1$ \\
\small {1. School of Mathematics and Statistics, Minnan Normal University, Zhangzhou, China}
}

\date{}

\newtheorem{lem}{Lemma}[section]
\newtheorem{thm}[lem]{Theorem}
\newtheorem{cor}[lem]{Corollary}

\newtheorem{pro}[lem]{Proposition}

\newtheorem{cla}{Claim}

\newtheorem{Prob}[lem]{Problem}

\begin{document}
\bibliographystyle{plain}
\newcommand{\udots}{\mathinner{\mskip1mu\raise1pt\vbox{\kern7pt\hbox{.}}
\mskip2mu\raise4pt\hbox{.}\mskip2mu\raise7pt\hbox{.}\mskip1mu}}
\maketitle
\begin{abstract}
A connected  graph $G$ with at least two vertices is {\em matching covered} if each of its edges lies in a perfect matching.
We say that an edge $e$ in a matching covered graph $G$ is {\em removable} if $G-e$ is matching covered. A pair $\{e,f\}$ of edges of a matching covered graph $G$ is a {\em removable doubleton} if $G-e-f$ is matching covered, but neither $G-e$ nor $G-f$ is. Removable edges and removable doubletons are called {\em removable classes}, introduced by Lov\'asz and Plummer in connection with ear decompositions of matching covered graphs.

A 3-connected graph is a {\em brick} if the removal of any two distinct vertices, the left graph has a perfect matching. A brick $G$ is {\em wheel-like} if $G$ has a vertex $h$, such that every removable class of $G$ has an edge incident with $h$.
Lucchesi and  Murty  proposed a problem of characterizing wheel-like bricks. We show that
every wheel-like brick may be obtained by splicing graphs whose underlying simple graphs are odd wheels in a certain manner.

A matching covered graph is {\em minimal} if the removal of any edge, the left graph is not matching covered. Lov\'asz and Plummer proved that the minimum degree of a minimal matching covered bipartite  graph different from $K_2$ is 2 by ear decompositions in 1977.
By the properties of wheel-like bricks,
 we prove that the minimum degree of a minimal matching covered graph other than $K_2$ is {either} 2 or 3.

\end{abstract}
\vskip 0.2in \baselineskip 0.1in
\section{Introduction}

Graphs considered in this paper may have multiple edges, but no loops. We follow \cite{BM08} for undefined notation and terminology.
Let $G$ be a graph with the vertex set $V(G)$ and the edge set $E(G)$. A connected nontrivial graph $G$ is {\em matching covered} if each of its edges lies in a perfect matching.
A graph $G$ with four or more vertices is {\em bicritical} if for any two distinct vertices $u$ and $v$ in $G$, $G-\{u, v\}$ has a perfect matching. Obviously,
every bicritical graph is matching covered.

For $X,Y\subseteq V(G)$, by $E_G[X,Y]$ we mean the set of edges of $G$ with {one end} in $X$ and the {other end} in $Y$. Let $\partial_G(X) = E_G[X,\overline{X}]$ be an edge cut of $G$, where $\overline{X} = V(G) \backslash X$. (If
$G$ is understood, the subscript $G$ is omitted.) If $X = \{u\}$, then we denote $\partial_G(\{u\})$, for
brevity, by $\partial_G(u)$ or $\partial(u)$.
{The the degree of $u$ in $G$, denoted by $d_G(u)$, is equal to the size of $\partial_G(u)$.}  Denote by $\delta(G)$ and $\Delta(G)$  the minimum degree  and the maximum degree of $G$, respectively.
An edge cut $\partial(X)$ is trivial if $|X| = 1$ or $|\overline{X}| = 1$.
Let $\partial(X)$ be an edge cut of $G$. Denoted by $G/(X\rightarrow x)$, or simply $G/X$, the graph obtained from $G$ by contracting $X$ to a singleton $x$ (and removing any resulting loops). The graphs $G/X$ and $G/\overline{X}$ are {the} $\partial(X)$-contractions of $G$.

An edge cut $\partial(X)$ is {\em separating} if both $\partial(X)$-contractions of $G$ are matching covered, and $\partial(X)$ is {\em tight} of $G$ if $|\partial(X)\cap M|=1$ for every perfect matching $M$ of $G$. Obviously, a trivial edge cut is a tight cut and every tight cut is separating.
{A matching covered graph is a {\em brick} if it is nonbipartite and every tight cut is trivial, and it is {\em solid} if every separating cut is a tight cut.}
Moreover, a  graph $G$ is a brick if and only if $G$ is 3-connected and bicritical \cite{ELP82}.
There is a procedure called a {\em tight cut decomposition}, due to Lov\'{a}sz \cite{Lovasz87}, which can be applied to $G$ to produce a list of unique bricks and braces (a matching covered bipartite graph in which every tight cut is trivial). We say that this list of bricks are the bricks of $G$. A matching covered graph $G$ is called a {\em near-brick} if $G$ contains only one brick. Obviously, every brick is a near-brick.

We say that an edge $e$ in a matching covered graph $G$ is {\em removable} if $G-e$ is matching covered. A pair $\{e,f\}$ of edges of a matching covered graph $G$ is a {\em removable doubleton} if $G-e-f$ is matching covered, but neither $G-e$ nor $G-f$ is. Removable edges and removable doubletons are called {\em removable classes}.
Lov\'asz \cite{lo} proved  that every brick different from $K_4$  (the complete graph with 4 vertices)  and
{the triangular prism (the complement of a cycle of length 6)}  has a removable edge. Improving Lov\'asz's result,  Carvalho, Lucchesi and Murty obtained a lower bound of removable classes of a brick in terms of the maximum degree.
\begin{thm}[\cite{CLM99}]\label{thm:re_in_brick}
    Every brick has at least $\Delta(G)$ removable classes. Moreover, every brick has at least $\Delta(G)-2$ removable edges.
\end{thm}

 For an integer $k\geq3$, the \emph{wheel} $W_k$ is the graph obtained from a cycle $C$ of length $k$ by adding a new vertex $h$ and joining it to all vertices of $C$.
The cycle $C$ is the \emph{rim} of $W_k$, the vertex $h$ is its \emph{hub}. Obviously, every wheel is planar. A wheel $W_k$ is odd if $k$ is odd. The graph $K_4$  is an odd wheel that every edge lies in a removable doubleton. For an odd wheel other than $K_4$,
it can be checked every edge on the rim is not removable, and  every edge incident with the hub is removable  (see Exercise 2.2.4 in \cite{Lucchesi2024}). We say that $G$ is {\em wheel-like} if for every removable class $R$ of $G$, there exists a vertex $h$ of $G$, called its {\em hub}, such that $|R\cap \partial(h)|=1$. Lucchesi and Murty proposed the following problem.
\begin{Prob}\label{prob:wheel-like}{\rm (see Page 216, and  Unsolved Problems 10 in \cite{Lucchesi2024})}
   Characterize wheel-like bricks;  characterize wheel-like bricks as a splicing of two bricks.
\end{Prob}

 We obtain the following theorem in this paper, where the graph family $\mathcal{G}$ is defined in Section 3.
\begin{thm}\label{thm:wheel-like}
    Let $G$ be a wheel-like brick. Then $G\in \mathcal{G}$.
\end{thm}

Let $G$ be a matching covered graph. We say that $G$ is {\em minimal} if $G-e$ is not a matching covered graph for any edge $e$ in $G$. Obviously, if $G$ is {minimal}, then $G$ has no removable edges. It is known that every matching covered graph on four or more vertices is 2-connected \cite{LP86}. So, the minimum degree of a matching covered graph with more than two vertices is at least 2. {Lov\'asz and Plummer \cite{LP77} proved that $\delta(G)=2$ if $G$ is a minimal matching covered bipartite graph \footnote{Lov\'asz and Plummer used the terminology ``minimal elementary bipartite graph". In fact, a graph is minimal elementary bipartite graph if and only if it is a minimal matching covered bipartite graph.}.
For example, a cycle with even number of vertices is a minimal matching covered bipartite graph with minimum degree 2. For a minimal matching covered nonbipartite  graph, the minimum degree may be 3; for example, $K_4$ and {the triangular prism}  are such graphs.
Using the properties of wheel-like bricks, we prove  the following result.

\begin{thm}\label{thm:main-thm}
    Let $G$ be a minimal matching covered graph with at least four vertices. Then $\delta(G)=2$  or 3.
\end{thm}
 We will present some known results in Section 2.
The propositions of wheel-like bricks
will be presented in Section 3, and the proof of Theorem \ref{thm:main-thm} will be given in Section 4.

\section{Preliminaries}
We begin with some notation. For a vertex  $u\in V(G)$, denoted by $N_G(u)$ or simply $N(u)$, the set of vertices in $G$ adjacent to $u$.
{A component with an odd (even) number of vertices is called an {\em odd (even) component}.}
We denote by $o(G)$ the number of components with odd number of vertices of a graph $G$.
 A nonempty vertex set $B$ of a graph $G$ that has a perfect matching is a {\em barrier} if $o(G-B)=|B|$. A component (or a  barrier) is trivial if it contains exactly one vertex.
 Tutte  proved the following theorem in 1947.
\begin{thm}[\cite{Tutte47}]\label{thm:Tutte}
    A graph $G$ has a perfect matching if and only if $o(G-S)\le|S|$, for every $S\subseteq V(G)$.
\end{thm}
Using Tutte's Theorem, we have the following properties about matching covered graphs.
\begin{cor}[\cite{Lovasz87}]\label{cor:M-C-without-even-components}
  Let $G$ be a matching covered graph and let $S$ be a subset of $V(G)$. Then, $o(G -S) \le |S|$, with equality only if $S$ is independent and $G-S$ has no even components.
\end{cor}

\begin{pro}[\cite{LP86}]\label{pro:mc_is_Bi}
    A matching covered graph $G$ different from $K_2$ is bicritical if and only if  every barrier  of $G$ {is trivial.}
\end{pro}

A vertex set $S$ of {a matching covered graph} $G$ is a {\em 2-separation} if $|S|=2$, $G-S$ is disconnected and each of the components of $G-S$ is even.
The following corollary can be gotten directly by Proposition \ref{pro:mc_is_Bi}.
\begin{cor}\label{cor:2-vtx_cut_of_Bi}
    Let $G$ be a bicritical graph different from $K_2$ and let $u,v\in V(G)$. If $G-\{u,v\}$ is disconnected, then $\{u,v\}$ is a 2-separation of $G$.
\end{cor}

Let $G$ be a matching covered graph.  If there exists a barrier $B$ of $G$ and an odd component $Q$ of $G-B$ such that $C=\partial(V(Q))$, we say  that the edge cut $C$  is a {\em barrier-cut} (associated with $B$).
Let $\{u,v\}$ be a 2-separation of $G$, and let us divide the components of $G-\{u,v\}$ into two nonempty subgraphs $G_1$ and $G_2$. The cuts $\partial(V(G_1)+u)$ and $\partial(V(G_1)+v)$ are both {\em 2-separation cuts} associated with $\{u,v\}$ of $G$.
Barrier-cuts and 2-separation cuts, which are tight cuts, play an important role during tight cut decomposition.
A barrier-cut $\partial(X)$ associated with a barrier $B$ of  $G$ is called a {\em special barrier-cut} if  $G[X]$ is the only one nontrivial odd component of $G-B$.
The following result can be gotten by the definition of near-brick directly (see Proposition 4.18 in \cite{Lucchesi2024} for example).
\begin{pro}[\cite{Lucchesi2024}]\label{pro:near-brick}
     Let $G$ be a near-brick. Then every tight cut of $G$ is a special barrier-cut.
\end{pro}

\subsection{ The splicing of two graphs and robust cuts}

Let $G$ and $H$ be two vertex-disjoint graphs and let $u$ and $v$ be vertices of $G$ and $H$, respectively, such that $d_G(u)=d_H(v)$.
Moreover, let $\theta$ be a given bijection between $\partial_H(v)$ and $\partial_G(u)$.
We denote by $(G(u)\odot H(v))_\theta$ the graph obtained from the union of $G-u$ and $H-v$ by joining, for each edge $e$ in $\partial_H(v)$, the end of $e$ in $H$ belonging to $V(H)-v$ to the end of $\theta(e)$ in $G$ belonging to $V(G)-u$;
and refer to $(G(u)\odot H(v))_\theta$ as the graph obtained by \emph{splicing $G$ (at $u$), with $H$ (at $v$), with respect to the bijection $\theta$}, for brevity, to $G(u)\odot H(v)$. We say that $u$ and $v$ are the splicing vertices of $G$ and $H$, respectively.
In general, the graph resulted from splicing two graphs $G$ and $H$ depends on the choice of $u$, $v$ and $\theta$. The following proposition can be gotten by the definition of matching covered graphs directly (see Theorem 2.13 in \cite{Lucchesi2024} for example).

\begin{pro}\label{thm:MC_IS_MC}
    The splicing of two matching covered graphs is also matching covered.
\end{pro}

{\begin{pro}[\cite{Lucchesi2024}]\label{pro:solid-MC}
    A matching covered graph is solid if and only if each of its bricks is solid.
\end{pro}}

Let $G$ be a matching covered graph.
A separating cut $C$ of $G$ is a {\em robust cut} if $C$ is not tight and both $C$-contractions of $G$ are near-bricks.

\begin{thm}[\cite{CLM05}]\label{thm:robust-cut}
    Every nonsolid brick $G$ has a robust cut $C$ such that one of the $C$-contractions of $G$ is solid.
\end{thm}

\begin{cor}\label{cor:robust-cut-solid-brick}
     Every nonsolid brick $G$ has a robust cut $\partial(X)$ such that there exists a subset $X'$ of $X$ and a subset $X''$ of $\overline{X}$ such that $G/\overline{X'}$ is a solid brick, $G/\overline{X''}$ is a brick and the graph $H$, obtained from $G$ by contracting $X'$ and  ${X''}$  to single vertices $x'$ and ${x''}$, respectively, is bipartite and matching covered, where  $x'$ and ${x''}$ lie in different color classes of $H$.
\end{cor}
\begin{proof}
     By Theorem \ref{thm:robust-cut}, assume that $\partial(X)$ is a robust cut of $G$ such that one of the $\partial(X)$-contractions of $G$ is solid.
    Let $G_1=G/(X\to x)$ and let $G_2=G/(\overline{X}\to \overline{x})$.
    Then $G_1$ and $G_2$ are near-bricks.
    Then every tight cut of $G_1$ and $G_2$ is a special barrier-cut by Proposition \ref{pro:near-brick}.
    For each $i\in\{1,2\}$, we may assume that $\partial(Y_i)$ is a special barrier-cut associated with a maximum barrier $B_i$ of $G_i$, such that $G[Y_i]$ is the only  nontrivial component of $G_i-B_i$.

    If $B_1$ is not trivial in $G_1$, then $x\in B_1$, as $B_1$ is not a barrier of $G$ (note that $G$ is a brick).
    Let $G_1'= G_1/(\overline{Y_1}\to\overline{y_1})$.
    If $G_1'$ is not a brick, then $G_1'$ has a nontrivial tight cut.
    By Proposition \ref{pro:near-brick} again, we may assume that $B_1'$ is a nontrivial barrier of $G_1'$ such that there exists a special barrier-cut associated with it.
    Then we have $\overline{y_1}\in B_1'$ (otherwise, it can be checked that $B_1'$ is also a nontrivial barrier of $G$).
    Thus,  it can be checked that $B_1\cup B_1'\setminus\{\overline{y_1}\}$ is a barrier of $G_1$, such that $|B_1|< |B_1\cup B_1'\setminus\{\overline{y_1}\}|$, as $B_1'$ is nontrivial. It contradicts the assumption that $B_1$ is a maximum barrier of $G_1$.
    Thus, $G_1'$ is a brick.
    If $B_1$ is trivial, then  $G_1$ is a brick (in this case, $B_1=\{x\}$).
    Similarly, $\overline{x}\in B_2$ and $\partial(Y_2)$ is a special barrier-cut of $G_2$, such that $G_2/(\overline{Y_2}\to\overline{y_2})$ is a brick.
    Note that one of $G_1$ and $G_2$, say $G_1$,  is solid. Then $G'_1$ is a solid brick by Proposition \ref{pro:solid-MC}.

    For each $i\in\{1,2\}$, let $H_i=G_i/(Y_i\to y_i)$.
    Then for each $i\in\{1,2\}$, $H_i$ is bipartite and matching covered (as $\partial(Y_i)$ is a special barrier-cut in $G_i$).
    Note that $x\in B_1$ and $\overline{x}\in B_2$, that is, $x\in V(H_1)$ and $\overline{x}\in V(H_2)$.
    Then $(G/Y_1)/Y_2=H_1(x)\odot H_2(\overline{x})$.
    Let $H=H_1(x)\odot H_2(\overline{x})$.
    We will complete the proof by showing that $H$ is a matching covered bipartite graph, and $y_1$ and $y_2$ lie in different color classes of $H$.
    Recall that $H_1$ and $H_2$ are bipartite and matching covered. Then $H$ is a matching covered bipartite graph by Proposition \ref{thm:MC_IS_MC}.
    As $y_1$ and $x$ lie in different color classes of $H_1$, and $y_2$ and $\overline{x}$ lie in different color classes of $H_2$, $y_1$ and $y_2$ lie in the different color classes of $H$. Therefore, the result holds by letting $X'=Y_1$ and $X''=Y_2$.
\end{proof}

\begin{lem}[\cite{LX2024}]\label{bipar-H-nonre}
    Let $\partial(X)$ and $\partial(Y)$ be two robust cuts of a brick $G$ such that  $G/{X}$ and $G/{Y}$ are bricks, and $(G/(\overline{X}\rightarrow\overline{x}))/(\overline{Y}\rightarrow\overline{y})$ is a bipartite graph $H$.
    Then every edge incident with $\overline{x}$ is removable in $H$.
\end{lem}

\subsection{Removable classes}

 We may assume that an edge $e$ is removable in a matching covered graph $G$  if $e\notin E(G)$. The following lemma is easy to verify by the definition (e.g., see Propositions 8.7 and 8.8 in \cite{Lucchesi2024}).
\begin{lem}\label{lem:re_also_re}
    Let $C$ be a separating cut of a matching covered graph $G$. If an edge $e$ is removable in both $C$-contractions of $G$ then $e$ is removable in $G$.
    Moreover,  if $C$ is tight, then an edge $e$ is removable in $G$ if and only if $e$ is removable in both $C$-contractions of $G$.
\end{lem}

\begin{lem}[Lemma 3.1 in \cite{CLM02II}]\label{lem:removable_doubleton}
    Let $C=\partial(X)$ be a separating cut but not a tight cut of a matching covered graph $G$ and let $H=G/\overline{X}$.
    Suppose that $H$ is a brick, and let $R$ be a removable doubleton of $H$.
    If $R\cap C=\emptyset$ or if the edge of $R\cap C$ is removable in $G/X$ then $R\setminus C$ contains an edge which is removable in $G$.
\end{lem}

We shall denote a bipartite graph $G$ with bipartition $(A,B)$ by $G[A,B]$.
 The following proposition which can be derived from the definition of bipartite matching covered graphs will be used in the following text.
\begin{pro}\label{pro:mc-emptyset}
    Let $G[A,B]$ be a matching covered graph. Assume that $X\subseteq V(G)$ such that $N(X\cap A)\subseteq X\cap B$ ($X$ is not necessary nonempty). Then $|X\cap A|\le|X\cap B|$. Moreover,
    $|X\cap A|=|X\cap B|$ if and only if {either $X=\emptyset$ or $X=V(G)$.}
\end{pro}

\begin{pro}[\cite{Lovasz87}]\label{pro:tight-cut-in-bipartiteMC}
    Let $G[A,B]$ be a matching covered  graph. An edge cut $\partial(X)$ of $G$ is tight if and only if $||X\cap  A|-|X\cap  B||=1$ and every edge of $\partial(X)$ is incident with a vertex of the larger one between the two sets, $X\cap  A$ and $X \cap B$.
\end{pro}

\begin{lem}[\cite{CLM15}]\label{lem:nonre-bi}
    Let $G[A,B]$ be a matching covered graph, and $|E(G)|\ge2$.
    An edge $uv$ of $G$, with $u\in A$ and $v\in B$, is not removable in $G$ if and only if there exist nonempty proper subsets $A_1$ and $B_1$ of $A$ and $B$, respectively, such that:\\
    {\rm 1)} the subgraph $G[A_1\cup B_1]$  is matching covered, and\\
    {\rm 2)} $u\in A_1$ and $v\in B\setminus B_1$, and $E[A_1,B\setminus B_1 ]=\{uv\}$.
\end{lem}

Let $G[A,B]$ be a matching covered graph with {at least 4 vertices} and let $X$ be a vertex set of $G$ such that $|X\cap A|=|X\cap B|$.
We say that $X$ is a {\em $P$-set} of $G$ if either $|E[X\cap A,\overline{X}\cap B]|=1$ or $|E[\overline{X}\cap A,X\cap B]|=1$. Obviously,  $\overline{X}$ is a $P$-set if $X$ is a  $P$-set.
A $P$-set $X$ of $G$ is {\em minimum} if for each $P$-set $Y$ in $G$ different from $X$, $|X|\le |Y|$.
By Lemma \ref{lem:nonre-bi}, for every nonremoveble edge of a bipartite matching covered graph, there exist {at least} two $P$-sets associated with it.

\begin{lem}\label{lem:minimum-P-set}
    Let $G[A,B]$ be a matching covered  graph with at least 4 vertices and $\delta(G)\ge3$. If $X$ is a minimum $P$-set of $G$, then every edge of $E(G[X])$ is removable in $G$.
\begin{proof}
Without loss of generality, assume that $|E[\overline{X}\cap A,X\cap B]|=1$. Let $\{ab\}=E[\overline{X}\cap A,X\cap B]$, where $a\in \overline{X}\cap A$ and $b\in X\cap B$.
If $|X|=2$, then $E(G[X])$ consists of some multiple edges, as $d(b)\ge3$. As multiple edges are removable, the result holds.
Now assume that $|X|\ge4$. 
Suppose that there exists a nonremovable edge $uv$ of $G$, such that $uv\in E(G[X])$.
Let $Y=X\cup \{a\}$. It can be checked by Proposition \ref{pro:tight-cut-in-bipartiteMC} that $\partial(Y)$ is a tight cut.
Noting that $\overline{X}$ is also a $P$-set, $|\overline{X}|\geq |{X}|>2$ by the minimality of $X$.
Let $G'=G/(\overline{Y}\to \overline{y})$ (note that $G'\neq G$).
Then $uv$ is also a nonremovable edge of $G'$ by Lemma \ref{lem:re_also_re}. So
there exists a $P$-set $Z$ associated with $uv$ in $G'$. Without loss of generality, assume that $E[Z\cap A,\overline{Z}\cap B]=\{uv\}$ and $u\in Z\cap A$.
Note that $V(G')=\{a,\overline{y}\}\cup X$ and $\{u,v\}\subset X$. If $\{a,\overline{y}\}\subset Z$ (the case is the same if  $\{a,\overline{y}\}\subset \overline{Z}$), then $\overline{Z}\subset X$ and  $\overline{Z}\neq X$, that is, $|\overline{Z}|<|X|$.  Note that $\overline{Z}$ is also a $P$-set of $G$. It contradicts the  assumption that $X$ is a minimum $P$-set of $G$.
So $|\{a,\overline{y}\}\cap Z|=1$.

  As $E[Z\cap A,\overline{Z}\cap B]=\{uv\}$, $a\neq u$, $\overline{y}\neq v$
  and $a\overline{y}\in E(G')$, we have $a\in \overline{Z}\cap A$ and $\overline{y}\in Z\cap B$.
Suppose that  $b\notin \overline{Z}\cap B$ or $b=v$. As $N_{G'}(a)=\{\overline{y},b\}$ and $\overline{y}\notin \overline{Z}$, we have  $N(\overline{Z}\cap B\setminus\{v\})\subseteq \overline{Z}\cap A\setminus\{a\}$.  Since $|\overline{Z}\cap B\setminus\{v\}|=|\overline{Z}\cap A\setminus\{a\}|$, we have $|\overline{Z}\cap B\setminus\{v\}|=|\overline{Z}\cap A\setminus\{a\}|=0$ by Proposition \ref{pro:mc-emptyset}. It means that $\overline{Z}=\{v,a\}$.
If $b\notin \overline{Z}\cap B$, then $N_{G'}(v)=\{u\}$, contradicting the fact that $G'$ is 2-connected.
If $b=v$, then $N_{G'}(v)=\{u,a\}$. As $uv$ and $ab$ are nonremovable in $G'$, we have $d_{G'}(v)=2$, contradicting the assumption that $\delta(G')\ge3$ (as $\delta(G)\ge3$).
Thus, $b\in \overline{Z}\cap B$ and $b\neq v$. Let $W=\overline{Z}\setminus\{a,b\}$. Then $|W\cap A|=|W\cap B|$ and $E[W\cap B,\overline{W}\cap A]=\{uv\}$ (as $N_{G'}(a)=\{\overline{y},b\}$). It means that $W$ is a $P$-set of $G'$ associated with $uv$. Moreover, as $\overline{y}\notin W$ and $\overline{W}=Z\cup\{a,b\}$,  $W$ is also a $P$-set of $G$ and $|W|\le |V(G')|-4$. Note that $|X|=|V(G')|-2$. Then $|W|<|X|$, contradicting the assumption that $X$ is minimum. Therefore, the result follows.
\end{proof}
\end{lem}
\begin{lem}\label{lem:bipart-M-C-at-most-1-nonre}
    Let $G[A,B]$ be a matching covered  graph with at least 4 vertices.
    Assume that every vertex of $A$ has degree at least 3.
    Then $G$ has two nonadjacent removable edges or
    there exist two vertices $u,v$ in $V(G)$ such that $v\in B$, $d(v)=2$, $d(u)\ge4$ and every edge of $\partial(u)$ is removable.
\end{lem}
\begin{proof}
     We will prove the result by induction on $|V(G)|$.
    If $|V(G)|=4$, then it can be checked that  the result follows, as every vertex of $A$ has degree at least 3 and at most two neighbors.
    Suppose that the result holds for $|V(G)|\le n$. Now we consider the case when $|V(G)|=n+2$, where $n$ is an even integer {at least} 4.

    Assume firstly that  there exists a vertex $v\in B$ such that $d_G(v)=2$. Let $N_G(v)=\{v_1,v_2\}$ and let $Y=\{v,v_1,v_2\}$. It can be checked by Proposition \ref{pro:tight-cut-in-bipartiteMC} that $\partial(Y)$ is a nontrivial tight cut of $G$, as $|V(G)|>4$. Let $G'=G/(\overline{Y}\to \overline{y})$ and let $G''=G/(Y\to y)$. Then $G'$ and $G''$ are bipartite and matching covered.
    As $d_G(v)=2$, $N_{G'}(v_i)\setminus\{v\}=\{\overline{y}\}$ (note that $|V(G')|=4$ and $G'$ is 2-connected) and every vertex of $N(v)$ has degree at least 3, we have $|E[\{v_i\},\{\overline{y}\}]|\ge2$, for each $i\in\{1,2\}$.
    It means that every edge of $\partial(\overline{y})$ is removable in $G'$. 
    By induction, $G''$ has two nonadjacent removable edges or there exists a vertex $w$, such that $d_{G''}(w)\ge4$ and every edge of $\partial_{G''}(w)$ is removable in $G''$.
    If $G''$ has two nonadjacent removable edges or the vertex $w\neq y$, then the result holds by Lemma \ref{lem:re_also_re}, as every edge of $\partial(\overline{y})$ is removable in $G'$.
    Assume that $w=y$. Then every edge of $\partial(Y)$ is removable in $G$ by Lemma \ref{lem:re_also_re} again. We can find two nonadjacent edges in $\partial(Y)$ as $G$ is 2-connected. So the result follows.

    Now assume that $\delta(G)\ge3$.
    Let $X$ be a minimum $P$-set of $G$.  Without loss of generality, assume that   $|E[\overline{X}\cap A,X\cap B]|=1$.
    Let $\{ab\}=E[\overline{X}\cap A,X\cap B]$, where $a\in \overline{X}\cap A$ and $b\in X\cap B$.
    If $|X|\ge4$, then by Lemma \ref{lem:minimum-P-set},
    $E(G[X])$ contains two nonadjacent removable edges of $G$, so the result holds.
    Assume that  $|X|=2$. Let $Z=X\cup\{a\}$.
    It can be checked by  Proposition \ref{pro:tight-cut-in-bipartiteMC} that $\partial(Z)$ is a nontrivial tight cut as $|V(G)|>4$ and  $|E[Z\cap B,\overline{Z}\cap A]|=\emptyset$ (note that $E[\overline{X}\cap A,X\cap B]=\{ab\}$).
    Let $G_1=G/(\overline{Z}\to \overline{z})$ and let $G_2=G/(Z\to z)$.
    As $d_{G}(a)\ge3$ and $E_G[\overline{X}\cap A,X\cap B]=\{ab\}$, we have $|E_G[\{a\},\overline{X}]|\ge2$.
    Then  $|E_{G_1}[\{a\},\{\overline{z}\}]|\ge2$, as $\overline{X}\setminus\{a\}\subseteq \overline{Z}$.
     It means that every  edge of $E_{G_1}[\{a\},\{\overline{z}\}]$ is removable    in $G_1$.
    Let $\{a'\}=X\setminus\{b\}$.
     As $\delta(G)\ge3$, $N_G(b)=\{a',a\}$ and  $|E_{G}[\{a\},\{b\}]|=1$, we have $|E_{G}[\{a'\},\{b\}]|>1$.
     Then $a'b$ is also removable in $G$ by Lemma \ref{lem:re_also_re}. 

    By induction, $G_2$ has two nonadjacent removable edges, or there exists a vertex $s$ such that $d_{G_2}(s)\ge4$ and every edge of $\partial(s)$ is removable in $G_2$. Recall that $a'b$ is removable in $G$.
    If $G_2$ has two nonadjacent removable edges or the vertex $s\neq z$, that is, there exists a removable edge $e$ of $G_2$ which lies in $E(G_2)\setminus\partial(z)$, then $e$ is also removable in $G$ by Lemma \ref{lem:re_also_re}. Moreover, $e$ and $a'b$ is nonadjacent. So the result holds in this case.
    Assume that $s=z$. Recall that every edge of $E_{G_1}[\{a\},\{\overline{z}\}]$ is removable in $G_1$ and every edge of $\partial_{G_2}(z)$ is removable in $G_2$. Then there exists a removable edge  of $G$ which is incident with $a$ by Lemma \ref{lem:re_also_re} again.  Together with $a'b$, $G$ has two nonadjacent removable edges. The result holds.
\end{proof}

\begin{lem}\label{lem:V(H)_re_edge}
    Let $\partial(X)$ and $\partial(Y)$ be two robust cuts of a brick $G$ such that  $G/{X}$ and $G/{Y}$ are bricks, and $(G/(\overline{X}\rightarrow\overline{x}))/(\overline{Y}\rightarrow\overline{y})$ is a matching covered bipartite  graph $H$.
    If $G/(X\rightarrow x)$ is wheel-like  such that $x$ is its hub and every edge of $\partial_{G/{X}}(x)$ belongs to {some} removable class of $G/{X}$, and $|N_H(\overline{x})|\ge 2$, then there exists a removable edge $e$ of $G$ such that both  ends of $e$ belong to $\overline{X}\cup N(\overline{X})\setminus\overline{Y}$.
\end{lem}
\begin{proof}
    Let $G'=G/(X\rightarrow x)$.
        Note that $|N_H(\overline{x})|\ge 2$.
    Assume that $\overline{x}b$ is an edge  in $H$ such that $b\neq \overline{y}$.
    Then $\overline{x}b$ is removable in $H$ by Lemma \ref{bipar-H-nonre}. Let $G/\overline{Y}=(G'(x)\odot H(\overline{x}))_{\theta}$.
    If $\theta(\overline{x}b)$ is removable in $G'$, then $\overline{x}b$ is also removable in $G/\overline{Y}$ by Lemma \ref{lem:re_also_re}.
    Since $\overline{x}b\notin \partial(\overline{y})$,
    $\overline{x}b$ is removable in $G$ by Lemma \ref{lem:re_also_re} again.
    Moreover, the ends of $\overline{x}b$ belong to $\overline{X}\cup N(\overline{X})\setminus\overline{Y}$. The result follows by setting $e=\overline{x}b$ in this case.
    If $\theta(\overline{x}b)$ is an edge of a removable doubleton in $G'$, assume that $\{e',\theta(\overline{x}b)\}$  is  a removable doubleton  in $G'$. Then $e'$ is a removable edge in $G/\overline{Y}$ by Lemma \ref{lem:removable_doubleton}.
    So $e'$ is removable in $G$ by Lemma \ref{lem:re_also_re} once more. As the ends of $e'$ belong to $\overline{X}\cup N(\overline{X})\setminus\overline{Y}$, the result follows by setting $e=e'$.
\end{proof}

\section{Wheel-like bricks}
In this section, we present some properties of  wheel-like bricks and prove the Theorem \ref{thm:wheel-like}. {First, we have the following propositions.}

\begin{pro}[\cite{Lucchesi2024}]\label{pro:solid-wl}
Let $G$ be a solid brick and let $h$ be a vertex of $G$. Either $G$ is a wheel having $h$ as a hub, possibly with multiple edges incident with $h$, or $G$ has two removable edges not incident with $h$.
\end{pro}

\begin{pro}[\cite{LX2024}]\label{pro:mutiedges-add}
    Let $G$ be a wheel-like brick and let $h$ is its hub. Then all the multiple edges are incident with $h$.
\end{pro}

A nonbipartite matching covered graph $G$ is {\em near-bipartite} if it has a pair of edges $e$ and $f$ such that the subgraph $G-e-f$  obtained by the deletion of $e$ and $f$ is a matching covered bipartite graph.
In fact, every brick with a removable doubleton is near-bipartite \cite{Lovasz87}.
\begin{thm}[Theorem 9.17 in \cite{Lucchesi2024}]\label{thm:sim-near-bi-brick}
      Every simple near-bipartite brick distinct from $K_4$, {the triangular prism} and $R_8$ (see Figure \ref{fig:r8}) has two nonadjacent removable edges.
\end{thm}

 \begin{figure}[!h]
    \centering
    \includegraphics[totalheight=1.8cm]{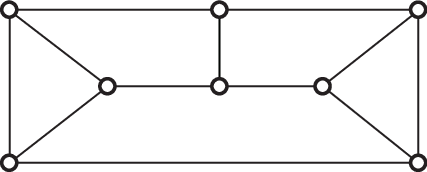}
    \caption{$R_8$.}
    \label{fig:r8}
\end{figure}

\begin{lem}[\cite{LX2024}]\label{lem:simple-nbb-wl-K4}
   {\rm 1)} Let $G$ be a simple near-bipartite brick. Then $G$ is wheel-like if and only if $G$ is isomorphic to $K_4$.

   {\rm 2)}  Let $G$ be a simple planar brick with six vertices. Then $G$ is a wheel-like brick if and only if $G$ is isomorphic to $W_5$.
\end{lem}

\begin{pro}[\cite{CLM06}]\label{pro:sim-wl-6-vtx}
    Let $G$ be a simple brick on six vertices. Then $G$ is either nonsolid or $W_5$.
\end{pro}

\begin{lem}\label{lem:wl-6vtx}
   Let $G$ be a wheel-like brick on 6  vertices and let  $h$ be the hub of $G$. Then 
    $G$ is isomorphic to $W_5$, possibly with multiple edges incident with $h$.
\end{lem}
\begin{proof}
    As $G$ is wheel-like, the only possible  multiple edges in $G$ are incident with $h$ by Proposition \ref{pro:mutiedges-add}.
    If $G$ is solid or planar, then the underlying simple graph of $G$ is isomorphic to $W_5$ by Proposition \ref{pro:sim-wl-6-vtx} and Lemma \ref{lem:simple-nbb-wl-K4}.  So the result is obviously.

    Suppose that $G$ is nonsolid and nonplanar. Then $G$ has a nontrivial separating cut $\partial(X)$.
    As $|V(G)|=6$, $|X|=3$ and $|\overline{X}|=3$. Note that $G$ is 3-connected and both $\partial(X)$-contractions are matching covered. Then both $\partial(X)$-contractions are isomorphic to $K_4$ (up to multiple edges).
    It can be checked that $G$ is isomorphic to {the triangular prism} or one of the graphs in Figure \ref{fig:C6-add-edges}. All the graphs in Figure \ref{fig:C6-add-edges} are not wheel-like (the bold edges are removable). By Lemma   \ref{lem:simple-nbb-wl-K4}, {the triangular prism} is not wheel-like. Therefore, the result holds.\end{proof}

 \begin{figure}[!h]
    \centering
    \includegraphics[totalheight=4.5cm]{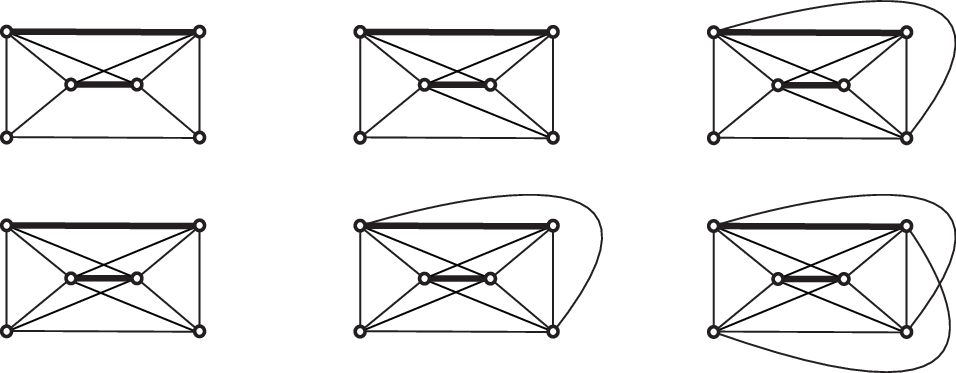}
    \caption{Nonplanar nonsolid bricks on six vertices, where the bold edges are removable.}
    \label{fig:C6-add-edges}
\end{figure}

\begin{lem}[\cite{LX2024}]\label{lem:at_least_one_wheel-like}
    Let $G_1$ and $G_2$ be two disjoint bricks and let $u\in V(G_1)$ and $v\in V(G_2)$.
    Assume that $G={G_1(u)\odot G_2(v)}$ is a brick.  \\
    {\rm 1)} If $G$ is wheel-like, then at least one of $G_1$ and $G_2$ is wheel-like such that $u$ or $v$ is its hub.\\
    {\rm 2)} If $G_1$ is  wheel-like with  $u$ as its hub, and every edge of $\partial_{G_1}(u)$ lies in some removable class of $G_1$, then $G_2$ is also wheel-like.
\end{lem}

Lu and Xue characterized wheel-like bricks that are obtained from the splicing of  two odd wheels.

\begin{lem}[\cite{LX2024}]\label{Wi_Wj}
    Let $G$ and $H$ be two odd wheels such that  $V(G)=\{u_h,u_1,u_2,\ldots, u_s\}$ and $V(H)=\{v_h,v_1,v_2,\ldots, v_t\}$, where $u_h$ and $v_h$ are the hubs of $G$ and  $H$ respectively. Assume that
     $u\in V(G)$, $v\in V(H)$, $d_G(u)=d_H(v)$, and  $G(u)\odot H(v)$ is a brick.
    The graph $G(u)\odot H(v)$ is wheel-like if and only if the following statements hold.\\

    {\rm 1)}. $|\{u,v\}\cap \{u_h,v_h\}|=1$. Without loss of generality, assume that $u=u_h$, that is $v \neq v_h$. Then $|V(G)|\ge 6$.

    {\rm 2)}. All the multiple edges of $G$ and $H$ are incident with $u_h$ and $v_h$, respectively.

    {\rm 3)}. Without loss of generality, assume that $v=v_t$  and $\{u_1v_1,u_rv_{t-1}\}\subset E(G(u)\odot H(v))$, where $1\leq r\leq s$. Then $r\neq 1$ and $u_1u_r\notin E(G)$.
\end{lem}

Let $\mathcal{K}_4^+$ be the family of graphs that  the underlying simple graph is isomorphic to $K_4$, and all the multiple edges have the same ends. Then every graph in $\mathcal{K}_4^+$ is wheel-like; every graph in $\mathcal{K}_4^+$, other than $K_4$, has exactly two hubs.
Let $\mathcal{G}_1$ be the family of wheel-like odd wheels {(an odd wheel having a hub, possibly with multiple edges incident with the hub).}
It should be noted that  every graph in $\mathcal{G}_1\setminus \mathcal{K}_4^+$  has exactly one hub, and all the multiple edges of it are incident with the hub. For an integer $j$ ($j>0$), let $\mathcal{G}_{j+1}$ be the family of graphs with at least 8 vertices  gotten by the splicing of one graph in $\mathcal{G}_j$ and one graph in $\mathcal{G}_{1}$, say $G_j$ and $H_j$, such that
\\
 1) if $H_j\cong K_4$, then $u_j\notin U(G_j)$; if $H_j\in\mathcal{K}_4^+\setminus\{K_4\}$, then $v_j\in U(H_j)$;   if $H_j\notin \mathcal{K}_4^+$, then $|\{u_j, v_j\}\cap (U(G_j)\cup U(H_j))|=1$; and \\
2) if $|V(H_j)|=4$ and $u_j\notin U(G_j)$, then for any nonremovable edge $e$ of $\partial(v_j)$ in $H_j$, the corresponding edge of $e$ (incident with $u_j$) in $G_j$ is not incident with any vertex of $U(G_j)$.\\
Where $U(G)$ is the set of vertices with maximum degree in $G$; $u_j$ and $v_j$ are the splicing vertices of $G_j$ and $H_j$, respectively.
{ Let $\mathcal{G}= \cup_i \mathcal{G}_i$.}

It should be noted that the maximum degree of any graph different from $K_4$ in $\mathcal{G}$ is at least 4.
The vertex of a graph in $\mathcal{G}$ with the maximum degree is called a {\em hub} of it.
Moreover, we have the following lemma.

\begin{lem}\label{lem:wheel-like}
Let $G\in\mathcal{G}$ and $|V(G)|>4$. Then the following statements hold.

{\rm 1)}. {$|U(G)|=1$.}

{\rm 2)}. Every edge incident with the hub of $G$ is removable in $G$.
\end{lem}
\begin{proof}

    Without loss of generality, assume that $G\in \mathcal{G}_{n}$ $(n\in \{1,2,3,\ldots\})$.
   We will prove the result by induction on $n$.
    If $n=1$, every graph in $\mathcal{G}_1$ is a wheel-like odd wheel with at least 6 vertices. So, 1) and 2) hold for this case.
    Suppose that the result holds when $n\le s$. Now we consider $G\in \mathcal{G}_{s+1}$, where $s\ge1$.
    Then $G=G_s(u_s)\odot H_s(v_s)$, where $G_s\in\mathcal{G}_s$ and $H_s\in\mathcal{G}_1$.
    We may assume that $|V(G_s)|\ge6$ (if $s=1$, by interchanging $G_1$ with $H_1$ if necessary, such that $|V(G_1)|\ge6$).
    Let $h_s\in U(G_s)$.

    As $|V(G_s)|\ge6$, by inductive hypothesis, we have $U(G_s)=\{h_s\}$.
    Obviously, $d_{G_s}(h_s)\ge4$.
    If $H_s\cong K_4$, then $u_j\notin U(G_j)$ and so, $h_s\in V(G)$. Hence, $U(G)=\{h_s\}$ (note that every vertex in $K_4$ is of degree 3).
    If $H_s\in \mathcal{K}_4^+\setminus\{K_4\}$, let $U(H_s)=\{t_s,t'_s\}$. Since $v_s\in U(H_s)$, without loss of generality, assume that $v_s=t_s$.
    If $u_s=h_s$, then $(U(G_s)\cup U(H_s))\cap V(G)=\{t'_s\}$ and so, $U(G)=\{t'_s\}$.
    If $u_s\neq h_s$, then $U(G)=\{h_s\}$, as $d_G(h_s)=d_{G_s}(h_s)> d_{G_s}(u_s)=d_{H_s}(v_s)$, $v_s=t_s\in U(H_s)$ and $h_s\in V(G)$.
    If $H_s\notin \mathcal{K}_4^+$, assume that $U(H_s)=\{t_s\}$.
    Then $|\{u_s,v_s\}\cap \{h_s,t_s\}|=1$.
    Recalling $U(G_s)=\{h_s\}$ and $U(H_s)=\{t_s\}$,  we have $|U(G)|=|(U(G_s)\cup U(H_s))\cap V(G)|=1$. Therefore, 1) holds.

    As $|V(G_s)|\ge6$, by inductive hypothesis, every edge of $\partial_{G_s}(h_s)$ is removable in $G_s$.
    If $|V(H_s)|\ge6$, then $|\{u_s,v_s\}\cap \{h_s,t_s\}|=1$.
    Assume that $u_s=h_s$ and $v_s\neq t_s$ (the case when {$u_s\neq h_s$ and} $v_s= t_s$ is similar).
    By the proof of 1) of this lemma, $U(G)=\{t_s\}$.
    Then by inductive hypothesis, every edge of $\partial_{H_s}(t_s)$ is removable in $H_s$.
    Therefore, every edge of $\partial_G(t_s)$ is removable in $G$ by Lemma \ref{lem:re_also_re}, and then 2) holds in this case.
    Now assume that $|V(H_s)|=4$, that is, the underlying simple graph of $H_s$ is isomorphic to $K_4$.

    If $H_s\in \mathcal{K}_4^+$ and $u_s\neq h_s$, then $U(G)=\{h_s\}$ by the proof of 1) of this lemma.
    Note that the corresponding edge (incident with $u_s$) of any nonremovable edge of $H_s$ is not incident with $h_s$  and every edge of $\partial_{G_s}(h_s)$ is removable in $G_s$. Then every edge of $\partial_G(h_s)$ is also removable in $G$ by Lemma \ref{lem:re_also_re}.

    If $H_s\in \mathcal{K}_4^+$ and $u_s= h_s$, then $d_{H_s}(v_s)=d_{G_s}(u_s)\ge4$ and so, $H_s\ncong  K_4$ and $\{u_s,v_{s}\}\subset \{h_s,t_s,t'_s\}$.
    Without loss of generality, assume that $v_s=t_s$. Then $U(G)=\{t'_s\}$.
    By Lemma \ref{lem:re_also_re}, every edge of $\partial_{H_s}(t'_s)$ which is removable in $H_s$ is also  removable in $G$.
    For any nonremovable edge $e$ of $\partial_{H_s}(t'_s)$ in $H_s$, there exists an edge $e'$ which is incident with $v_s$ in  $H_s$, such that $\{e,e'\}$ is a removable doubleton of $H_s$.
    By Lemma \ref{lem:removable_doubleton}, $e$ is removable in $G$. So every edge of $\partial_G(t'_s)$ is also removable in $G$.

    Now assume that $H_s\notin\mathcal{K}_4^+$, that is $|V(H_s)|=4$ and $H_s$ contains multiple edges with exactly one common vertex. Then $U(H_s)=\{t_s\}$. By inductive hypothesis, $|\{u_s,v_s\}\cap \{h_s,t_s\}|=1$.
    If $u_s=h_s$, then $U(G)=\{t_s\}$ by the proof of 1) of this lemma. {By Lemma \ref{lem:removable_doubleton},} every edge of $\partial_G(t_s)$ is removable in $G$.
    If $v_s=t_s$, then $U(G)=\{h_s\}$ by the proof of 1) of this lemma again.
    As $H_s\notin\mathcal{K}_4^+$, $v_s$ is incident with at most one nonremovable edge.
    Let $f$ be the only possible nonremovable edge of $\partial_{H_s}(v_s)$ in $H_s$. Then the corresponding edge of $f$ (incident with $u_s$) in $G_s$ is not incident with $h_s$. So, every edge of $\partial_G(h_s)$ is removable in $G$. Therefore, 2) holds.
    \end{proof}

\noindent{\bf {Proof of Theorem 1.3.}}
    If $G$ is solid, then by Proposition \ref{pro:solid-wl},  {$G$ is a wheel-like odd wheel. So $G\in\mathcal{G}_1$, the result holds. }
    Now we assume that $G$ is nonsolid.
    We will prove the result by induction on $|V(G)|$.
    By Lemma \ref{lem:wl-6vtx}, the result holds when $|V(G)|=6$.
    Now we assume that the result holds for every wheel-like brick with at most $n$ vertices ($n$ is even and {$n\ge6$}). In the following we will consider the case when $|V(G)|=n+2$.

    As $G$ is nonsolid, by Corollary \ref{cor:robust-cut-solid-brick}, there exist a subset $X'$ of $\overline{X}$ and a subset $X''$ of $X$ such that $G/(\overline{X'}\rightarrow\overline{x'})$ is a brick, $G/(\overline{X''}\rightarrow\overline{x''})$ is a solid brick, and $(G/({X'}\rightarrow {x'}))/({X''}\rightarrow {x''})$ is a  matching covered bipartite graph such that $x'$ and $x''$ lie in different color classes of $(G/({X'}\rightarrow {x'}))/({X''}\rightarrow {x''})$.
    Let $G'=G/\overline{X'}$, $G''=G/\overline{X''}$ and  $H=(G/({X'}\rightarrow {x'}))/({X''}\rightarrow {x''})$.
    Then every edge of $\partial_H({x'})$ and $\partial_H({x''})$ is removable in $H$ by Lemma \ref{bipar-H-nonre}.

    \begin{cla}
        $G'$ and $G''$ are wheel-likes.
    \end{cla}
    \begin{proof}
    By Theorem \ref{thm:re_in_brick}, $G'$ and $G''$ contain at least three removable classes, respectively.
    Suppose that there exist removable classes $R_1$ and $R_2$ in $G'$ and $G''$ respectively, such that $R_1\cap\partial(\overline{x'})=\emptyset$ and $R_2\cap\partial(\overline{x''})=\emptyset$.
    Then there exist two edges, $e_1$ and $e_2$, such that $e_1\in R_1$, $e_2\in R_2$ and both of $e_1$ and $e_2$ are removable in $G$ by Lemmas \ref{lem:re_also_re} and \ref{lem:removable_doubleton}.
    Note that $e_1$ and $e_2$ are nonadjacent, contradicting the assumption that $G$ is wheel-like.
    So at least one of $G'$ and $G''$ is wheel-like such that $\overline{x'}$ or $\overline{x''}$ is its hub.

    Without loss of generality, assume that $G'$ is wheel-like with $\overline{x'}$ as its hub.
    {By inductive hypothesis, $G'\in\mathcal{G}$ and then, every edge incident with $\overline{x'}$ belongs to some removable class (more exactly, every edge of $\partial(\overline{x'})$ is removable unless the underlying simple graph of $G'$ is isomorphic to $K_4$ by Lemma \ref{lem:wheel-like}).}
    Note that every edge incident with $x'$ is removable in $H$.
    If every removable class of $G''$ contains an edge incident with $\partial(\overline{x''})$, then $G''$ is wheel-like obviously.
    So we assume that there exists a removable class $R_3$ in $G''$ such that $R_3\cap \partial(\overline{x''})=\emptyset$. $|R_3|$ may be 1 or 2. In each case,
    by Lemmas \ref{lem:re_also_re} and \ref{lem:removable_doubleton}, $R_3$ contains a removable edge $e_0$ in $G$.
    Suppose that $|N_H(x')|\ge2$. Then there exists a removable edge $e$ in $G$ such that both ends of $e$ belong to $X'\cup N(X')\setminus{X''}$ by Lemma \ref{lem:V(H)_re_edge}.
    So $e$ and $e_0$ are two nonadjacent removable edges in $G$, contradicting the assumption that $G$ is wheel-like.
    Therefore, we have $|N_H(x')|=1$. {It means that $H$ is not 2-connected. Then the underlying simple graph of $H$ is $K_2$, as $H$ is matching covered (hence $H$ is connected).
    So $V(H)=\{x',x''\}$.}
    So $G$ is isomorphic $G'(\overline{x'})\odot G''(\overline{x''})$.
    By Lemma \ref{lem:at_least_one_wheel-like}, $G''$ is wheel-like since $G$ is wheel-like.
    \end{proof}

     By inductive hypothesis and Claim 1, $G',G''\in\mathcal{G}$.
    Let $h'$ and $h''$ be hubs of $G'$ and $G''$, respectively.
    By the proof of Claim 1, without loss of generality, assume that $\overline{x'}=h'$.
    Then we have the following claim.

    \begin{cla}
        $V(H)=\{x',x''\}$.
    \end{cla}
    \begin{proof}
        Suppose to the contrary that $|V(H)|\ge4$ (note that $H$ is bipartite and matching covered).
        Then $|N_H(x')|\ge2$ as $H$ is 2-connected.
        Note that $\overline{x'}=h'$ and every edge of $\partial_{G'}(h')$ belongs to some removable class of $G'$ (in fact, every edge of $\partial_{G'}(h')$ is removable in $G'$ unless the underlying simple graph of $G'$ is isomorphic to $K_4$ by Lemma \ref{lem:wheel-like}). Then by Lemma \ref{lem:V(H)_re_edge}, there exists a removable edge $e_1$ of $G$, such that both ends of $e_1$ belong to $X'\cup N(X')\setminus X''$.

        If $\overline{x''}=h''$, then by Lemma \ref{lem:V(H)_re_edge} again, there exists a  removable edge  $e_2$ of $G$, such that both ends of $e_2$ belong to $X''\cup N(X'')\setminus X'$, as $|N_H(x'')|\ge2$ and every edge of $\partial_{G''}(h'')$ belongs to some removable class of $G''$.
        Then $e_1$ and $e_2$ are two nonadjacent removable edges of $G$, contradicting the assumption that $G$ is wheel-like.
        So we assume that $\overline{x''}\neq h''$. As $G''$ is a brick, we have $\partial_{G''}(h'')\setminus\partial_{G''}(\overline{x''})\neq\emptyset$. Let $e\in \partial_{G''}(h'')\setminus\partial_{G''}(\overline{x''})$.
        Note that $e$ belongs to some removable class of $G''$ (as $e\in \partial_{G''}(h'')$) and $e\notin \partial_{G''}(\overline{x''})$.
        If $e$ is a removable edge in $G''$, then $e$ is also a removable edge in $G$ by Lemma \ref{lem:re_also_re}; and if $e$ lies in a removable doubleton of $G''$, then one edge in this removable doubleton is removable  in $G$ by Lemma \ref{lem:removable_doubleton}.
        So there exists a removable edge $e_3$ of $G$ (it is possible that $e_3=e$), such that both ends of $e_3$ belong to $X''$.
        Hence, $e_1$ and $e_3$ are two nonadjacent removable edges in $G$. This is a contradiction.
        Therefore, we have $|V(H)|=2$, that is, $V(H)=\{x',x''\}$.  
    \end{proof}

    By Claim 2, we have $G=G'(\overline{x'})\odot G''(\overline{x''})$.
      As $|V(G)|>6$, at least one of $|V(G')|$ and $|V(G'')|$ is at least 6.
    Moreover, we have the following claim.

\begin{cla}
Assume that $G'$ has at least 6 vertices and the underlying simple graph of $G''$ is isomorphic to $K_4$.
Then  the following statements hold.

{\rm 1)}
If $G''\cong K_4$, then $\overline{x'}\neq h'$; if $G''\in\mathcal{K}_4^+\setminus\{K_4\}$, then $\overline{x''}\in U(G'')$; if $G''\notin\mathcal{K}_4^+$, then $|\{\overline{x'},\overline{x''}\}\cap \{h',h''\}|=1$.

{\rm 2)} If $\overline{x'}\neq h'$, then for any nonremovable edge $e$ of $\partial(\overline{x''})$ in $G''$, the corresponding edge of $e$ (incident with $\overline{x'}$) in $G'$ is not incident with $h'$.
\end{cla}
\begin{proof}
     By Claim 1, $G'$ and $G''$ are wheel-like bricks.
    If $G''\cong K_4$, then $\overline{x'}\neq h'$, as $d_{G'}(h')\ge 4$ and every vertex of $V(G'')$ is of degree 3.
    Now assume that $G''\ncong K_4$.
    As $G$ is wheel-like,  by Lemma \ref{lem:at_least_one_wheel-like}, $|\{\overline{x'},\overline{x''}\}\cap (\{h'\}\cup U'')|\ge1$.
    If $G''\in \mathcal{K}_4^+\setminus \{K_4\}$, then $\overline{x''}\in U(G'')$, as $d_{G'}(h')\ge4$ and every vertex of $V(G'')\setminus U(G'')$ is of degree 3 in $G''$.
    If $G''\notin \mathcal{K}_4^+$, then $U(G'')=\{h''\}$ and $h''$ is incident with at least two removable edges (with precisely one common vertex $h''$) of $G''$.
    Note that $V(G')\setminus\overline{x'}=X$ and $V(G'')\setminus\{\overline{x''}\}=\overline{X}$.
    If $\overline{x'}=h'$ and $\overline{x''}=h''$, then by Lemma \ref{lem:re_also_re}, there exist two nonadjacent removable edges of $G$ lying in $\partial_G(X)$, as every edge of $\partial(h')$ is removable in $G'$ (by inductive hypothesis and Lemma \ref{lem:wheel-like}) and $\partial_{G}(X)$ is a robust cut of $G$.  It contradicts the assumption that $G$ is wheel-like.
    Therefore, $|\{\overline{x'},\overline{x''}\}\cap \{h',h''\}|=1$.

    Suppose, to the contrary, that there exists a nonremovable edge $e$ of $\partial(\overline{x''})$ in $G''$, and the corresponding edge of $e$ (incident with $\overline{x'}$) in $G'$, say $e'$, is incident with $h'$.
    As the underlying simple graph of $G''$ is isomorphic to $K_4$, there exist two vertices $s$ and $t$ in $V(G'')\setminus\{\overline{x''}\}$, such that $\{e,st\}$ is a removable doubleton of {the underlying simple graph of} $G''$. {Then either $|E[\{s\},\{t\}]|\ge2$ in $G''$ or $\{e,st\}$ is a removable doubleton of $G''$.}
    Note that every edge of $\partial(h')$ is removable in $G'$. {By Lemmas  \ref{lem:re_also_re} and \ref{lem:removable_doubleton},}  $st$ is removable in $G$ {which is not incident with $h'$.}
    {On the other hand, as every edge of $\partial (h')$ is removable in $G'$, every edge of $\partial (h')\setminus \partial (\overline{x'})$ is removable in $G$ by Lemma \ref{lem:re_also_re}.  As $\partial (h')\setminus \partial (\overline{x'})$ contains at least two adjacent edges (with one common vertex $h'$) and $G$ is wheel-like, every removable of $G$ is incident with $h'$. This is a contradiction. So 2) holds.}
\end{proof}

    By Claim 3, the result holds if the underlying simple graph of $G'$ or $G''$ is isomorphic to $K_4$.
    So assume that $|V(G')|\ge6$ and $|V(G'')|\ge6$.
    As $G',G''\in\mathcal{G}$ (by inductive hypothesis), each of $G'$ and $G''$ has exactly one hub by 1) of Lemma \ref{lem:wheel-like}.
    As $G$ is wheel-like and both of $G'$ and $G''$ are bricks, we have $|\{\overline{x'},\overline{x''}\}\cap \{h',h''\}|\ge1$ by Lemma \ref{lem:at_least_one_wheel-like}.
    Suppose to the contrary that $|\{\overline{x'},\overline{x''}\}\cap \{h',h''\}|=2$.
    By 2) of Lemma \ref{lem:wheel-like}, every edge of $\partial_{G'}(\overline{x'})$ and $\partial_{G''}(\overline{x''})$ is removable in $G'$ and $G''$, respectively.
    Then by Lemma \ref{lem:re_also_re}, every edge of $\partial_{G}(V(G')\setminus\{\overline{x'}\})$ is removable in $G$.
    As $G$ is a brick and $\partial_G(V(G')\setminus\{\overline{x'}\})$ is a nontrivial edge cut of $G$, there exist at least two nonadjacent edges of $\partial_G(V(G')\setminus\{\overline{x'}\})$, contradicting the assumption that $G$ is wheel-like.
    Therefore, $|\{\overline{x'},\overline{x''}\}\cap \{h',h''\}|=1$.
    As $G''$ is a solid brick, by Claim 1 and Proposition \ref{pro:solid-wl}, $G''$ is a wheel-like odd wheel, that is, $G''\in \mathcal{G}_1$. Therefore, $G\in\mathcal{G}$. $\hfill\square$

It should be noted that not every graph in $\mathcal{G}$ is wheel-like.  Lemma \ref{Wi_Wj}
will help to determine when the splicing of two odd wheels is wheel-like. When $n>2$, some edge not {incident with} the hub of $\mathcal{G}_n$ will be removable, even if the splicing between any two odd wheels satisfying the condition in  Lemma \ref{Wi_Wj} (see Figure \ref{fig:counter-exp} for example).
\begin{figure}[h]
  \centering
  \includegraphics[width=0.4\textwidth]{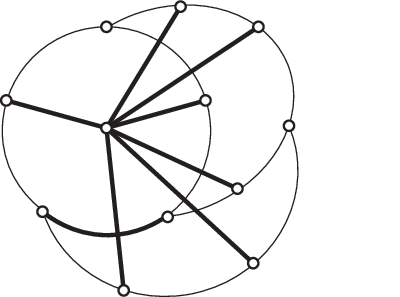}
  \caption{A brick in $\mathcal{G}_3$  which is not wheel-like (the bold edges are removable).}\label{fig:counter-exp}
\end{figure}
\begin{pro}\label{pro:brick-cubic-re}
    Let $G$ be a brick such that every removable edge of it is incident with a vertex $h$. Then every edge of $\partial(h)$ is removable or there exists a  vertex $u\in V(G)\setminus\{h\}$ such that $d_{G}(u)=3$.
        \begin{proof}
      If $G$ has a removable doubleton, then the underlying simple graph $H$ of $G$ is isomorphic to $K_4$, {the triangular prism} or $R_8$ by Theorem \ref{thm:sim-near-bi-brick}.
      If $G=H$, then the result holds as every vertex of $K_4$, {the triangular prism} and $R_8$ is of degree 3.
      Assume that $G$ has a multiple edge. Then every multiple edge of $G$ is incident with $h$ and so, $d_{G}(h)\ge4$.
     It can be checked that the result holds when $H\cong K_4$.
    {Assume that $H$ is the triangular prism} or $H\cong R_8$.
      {As $V(G)\setminus(\{h\}\cup N_G(h))\neq\emptyset$, we may assume that $v\in  V(G)\setminus(\{h\}\cup N_G(h))$. As every multiple edge of $G$ belongs to $E[\{h\}, N_G(h)]$ and every vertex of $H$ has degree 3, we have $d_G(v)=3$.}

      If $G$ has no removable doubletons, then $G$ is wheel-like. Note that  every edge incident with the hub in  an odd wheel with at least 6 vertices is removable.
      Then the result holds by Lemma \ref{lem:simple-nbb-wl-K4}, Lemma \ref{lem:wheel-like} and Theorem \ref{thm:wheel-like}.
      \end{proof}
\end{pro}

\begin{pro}[\cite{LP86}]\label{pro:2-sep-cut-of-bi}
    {Let $G$ be a bicritical graph and let $C$ be a 2-separation cut of $G$. Then both of the two $C$-contractions of $G$ are bicritical.}
\end{pro}

\begin{pro}[\cite{ZWY2022}]\label{pro:WL-bicritical-no-re}
    Let $G$ be a bicritical graph without removable edges. Then $G$ has at least four vertices of degree three. As a consequence, every bicritical graph with minimum degree at least 4 has removable edges.
\end{pro}

\begin{lem}\label{lem:WL-bicritical}
    Let $G$ be a bicritical graph with a removable edge. Assume that every removable edge of $G$ is incident with a vertex $h$. Then every edge of $\partial(h)$ is removable or  there exists a  vertex $s$ in $V(G)\setminus\{h\}$ such that $d(s)=3$.
   \end{lem}
\begin{proof}If $G$ is 3-connected, that is, $G$ is a brick, then the result holds by Proposition \ref{pro:brick-cubic-re}.
So we assume that $G$ is not 3-connected. As $G$ is bicritical,
     $G$ has a 2-separation  by Corollary \ref{cor:2-vtx_cut_of_Bi}.
    Let a bicritical graph $G$ be chosen with $|V (G)|$ minimum  such that some edge of $\partial(h)$ is not removable and every vertex in $V(G)\setminus\{h\}$ is of degree at least 4.

    As $K_4$ is a brick, we have $|V(G)|\ge6$.
    Let $\{u,v\}$ be a 2-separation of $G$. Assume that $\partial(X)$ is a 2-separation cut associated with $\{u,v\}$ such that $u\in X$.
    Let $H_1=G/(\overline{X}\to \overline{x})$ and $H_2=G/(X\to x)$.
    Then by Proposition \ref{pro:2-sep-cut-of-bi}, $H_1$ and $H_2$ are bicritical. Without  loss of generality, assume that $h\in X$.

 We first suppose that  $H_2$ contains no removable edges. Then  $V(H_2)\setminus\{x,v\}$ contains a vertex $s_1$ of degree 3 by Proposition \ref{pro:WL-bicritical-no-re}.
 So $d_G(s_1)=d_{H_2}(s_1)=3$.
 Now we suppose that  $H_2$ contains a removable edge.
 By Lemma \ref{lem:re_also_re}, the possible removable edges of $H_2$ are incident with the vertex $x$ (in this case, $h=u$), or belong to $E[\{x\},\{v\}]$. Then all the removable edges of $H_2$ are incident with $x$. As $|V(H_2)|<|V(G)|$, $V(H_2)\setminus \{x\}$  contains a vertex  of degree 3 or every edge of $\partial(x)$ is removable in $H_2$ by the minimality of $V(G)$.
     If $V(H_2)\setminus \{x\}$  contains a vertex  of degree 3, then this vertex, different from $h$, is also of degree 3 in $G$. Now we consider the case when  every edge of $\partial(x)$ is removable in $H_2$.  If every edge of $\partial_{H_1}(u)$ is removable in $H_1$, then every edge of $\partial_{G}(u)$ is removable in $G$ by Lemma \ref{lem:re_also_re}.
     So we assume that some edge of $\partial_{H_1}(u)$ is not removable in $H_1$. If every edge of $\partial_{H_1}(u)$ is not removable in $H_1$, then
     $H_1$ contains no removable edges. Similar to the case when  $H_2$ contains no removable edges, the result follows.
     So $\partial_{H_1}(u)$ contains removable edges and nonremovable edges  in $H_1$.
     Then every removable of  $H_1$ is incident with $u$.
As $|V(H_1)|<|V(G)|$, $V(H_1)\setminus \{u\}$  contains a vertex $s_2$ of degree 3  by the minimality of $V(G)$. Then $d_G(s_2)=d_{H_1}(s_2)=3$. The result follows.
   \end{proof}

\section{Proof of  Theorem \ref{thm:main-thm}}
Let $B$ be a maximal nontrivial barrier of a matching covered nonbipartite graph $G$.
Denote by $H(G,B)$ the graph obtained from $G$ by contracting every nontrivial odd component of $G-B$ to a singleton. By Corollary \ref{cor:M-C-without-even-components},
$B$ is an independent set and $G-B$ has no even components. So $H(G,B)$ is a bipartite graph with $B$ as one of its color classes.
Let $I=V(H(G,B))\setminus B$. {When no confusion arises, we assign the same label to any vertex (or edge) common to both $G$ and $H$.}
Note that $v\in V(G)\cap V(H)$ if and only if $v$ is not gotten by contracting a nontrivial odd component of $G-B$.
 Let $W_{H(G,B)}=\{u\in I: \mbox{$u$ is incident with some removable edge of $H(G,B)$}\}$.
As each odd component of $G-B$ is a shore of a barrier cut of $G$, $H(G,B)$ is a matching covered bipartite graph.

\hspace{-0.8cm}
{\bf Theorem 1.4.} {\em
    Let $G$ be a minimal matching covered graph with at least four vertices. Then $\delta(G)=2$  or 3.
}
 \begin{proof}

By Propositions \ref{pro:mc_is_Bi} and \ref{pro:WL-bicritical-no-re}, if a minimal matching covered graph has no nontrivial barriers, then the result holds.
So we consider the case when the minimal matching covered graph has a nontrivial barrier. Let a minimal matching covered $G$ that has minimum degree  at least $4$ be chosen with $|V (G)|$ minimum. {Then $G$ is nonbipartite.
 If $|V(G)|=4$, then $G$ is isomorphic to $K_4$, in which every vertex has degree 3.} So $|V(G)|>4$.
 We have the following claim.\\

\hspace{-0.7cm} {\bf Claim A.} {\em
For any {nontrivial barrier} $B$ of $G$, we have $|W_{H(G,B)}|\ge2$. Moreover,  for every vertex $u\in W_{H(G,B)}$, $u\notin V(G)$.
}
\begin{proof}
If $\delta(H(G,B))\ge3$, by Lemma \ref{lem:bipart-M-C-at-most-1-nonre}, $H(G,B)$ has two nonadjacent removable edges and then, $|W_{H(G,B)}|\ge2$.  So we will show that $\delta(H(G,B))\ge3$.
Suppose, to the contrary, that  there exists a vertex $k\in H(G,B)$ such that $d_{H(G,B)}(k)=2$ (As $H(G,B)$ is matching covered with at least 4 vertices, $H(G,B)$ is 2-connected). Then $k\in I$ and $k\notin V(G)$, as $B\subset V(G) $ and $\delta(G)\ge4$. Assume that $k$ is obtained by contracting the odd component $K$ of $G-B$. Let $H_1=G/(\overline{K}\to\overline{k})$.
Then $H_1$ is  matching covered as $\partial(V(K))$ is a barrier cut of $G$.
Note that every edge of $E(H_1)\setminus\partial_{H_1}(\overline{k})$ is not removable in $H_1$ by Lemma \ref{lem:re_also_re}, as $G$ has no removable edges.
As  $d_{H_1}(\overline{k})=2$, it can be checked that  every edge of $\partial_{H_1}(\overline{k})$ is not removable in $H_1$. Then
$H_1$ is minimal matching covered.
{Let $Y=\{\overline{k},k_1,k_2\}$, where $k_1$ and $k_2$ are the neighbors of $\overline{k}$ in $H_1$.
  Then $\partial_{H_1}(Y)$ is a tight cut of $H_1$. Let $H_2=H_1/(Y\to y)$. As $\{k_1,k_2\}\subset V(G)$, the degrees of $k_1$ and $k_2$  are at least 4, respectively. Moreover, for each $i\in\{1,2\}$, $|E[\{k_i\},\{\overline{y}\}]|\ge2$ in $H_1/(\overline{Y}\to\overline{y})$ and then, every edge incident with $\overline{y}$ is removable in $H_1/(\overline{Y}\to\overline{y})$.
   Since $H_1$ contains no removable edges, every edge incident with $y$ is not removable in $H_2$ by Lemma \ref{lem:re_also_re}.
   Note that $d_{H_2}(y)\ge4$ and $V(H_2)\setminus \{y\}\subset V(G)$. Then $\delta(H_2)\ge4$.
    Therefore, $H_2$ is a minimal matching covered graph with $\delta(H_2)\ge4$ and $|V(H_2)|<|V(G)|$,} contradicting the choice that $|V(G)|$ is minimum. So $\delta(H(G,B))\ge3$.

Suppose, to the contrary, that there exists a vertex $u\in V(G)\cap W_{H(G,B)}$. Since $H(G,B)$ is bipartite, we may assume that  $ub$ is removable in $H(G,B)$ where $b\in B$.
  As $u,b\in V(G)$, $ub$ is removable in $G$ by Lemma \ref{lem:re_also_re}, contradicting the assumption that $G$ has no removable edges. So $ V(G)\cap W_{H(G,B)}=\emptyset$.
\end{proof}
Let $B_0$ be a maximal nontrivial barrier of $G$. By Claim A, we may choose  a nontrivial odd component of $G-B_0$, say
 $Q_0$, such that $q_0\in W_{H(G,B_0)}$, where the vertex $q_0$ is gotten by contracting the component $Q_0$.
Let  $G_1=G/(\overline{V(Q_0)}\to \overline{q_0})$.
Then $G_1$ is not bipartite. Otherwise, suppose that $A'$ and $B'$ are two color classes of $G_1$ such that $\overline{q_0}\in B'$. As  $|V(G_1)|\ge4$, $B_0\cup B'\setminus\{\overline{q_0}\}$ is a barrier of G satisfying $|B_0\cup B'\setminus\{\overline{q_0}\}|>|B_0|$, which contradicts the maximality of $B_0$.
If $G_1$ is not bicritical, assume that $B_1$ is a maximal nontrivial barrier of $G_1$.
We have the following claim.\\

\hspace{-0.7cm} {\bf Claim B.}  {\em
There exists a nontrivial odd component $Q_1$ of $G_1-B_1$ such that $\overline{q_0}\notin Q_1$ $($i.e., $V(Q_1)\subset V(G)$$)$, and $q_1\in W_{H(G_1,B_1)}$, where $q_1$ is obtained by contracting the odd component $Q_1$ of $G_1-B_1$.}

\begin{proof} As $B_0$ is maximal and $B_1$ is not trivial, $\overline{q_0}\notin B_1$. Otherwise, $B_0\cup B_1\setminus\{\overline{q_0}\}$ is a barrier of $G$ satisfying $|B_0\cup B_1\setminus\{\overline{q_0}\}|>|B_0|$, which contradicts the maximality of $B_0$. So $\overline{q_0}\in V(H(G_1,B_1))\setminus B_1 $ or $\overline{q_0}$ lies in some component of $G_1-B_1$.
Let $q$ be a vertex in $V(H(G_1,B_1))\setminus B_1$ such that $q= \overline{q_0}$ if $\overline{q_0}\in V(H(G_1,B_1))\setminus B_1$, otherwise $q$ is the vertex gotten by contracting the nontrivial odd components of $G_1-B_1$ that contains $\overline{q_0}$.
 {As $\overline{q_0}\notin B_1$, $B_1$ is also a barrier of $G$. Then $H(G_1,B_1)=H(G,B_1)$.} Then similar to the proof of Claim A, we can show that $\delta(H(G_1,B_1))\ge3$ by the minimality of $G$. By Lemma \ref{lem:bipart-M-C-at-most-1-nonre}, we have $|W_{H(G_1,B_1)}|\ge2$.  So there exists a vertex $q_1$ in $W_{H(G_1,B_1)}$ such that  $q_1\neq q$. Then the nontrivial odd components of $G_1-B_1$ contracted to $q_1$ is that we need. So the claim holds.
\end{proof}

It is known that all the maximal barriers in $G$ is a partition of $V(G)$ (see Lemma 5.2.1 in \cite{LP86}). We may contract several maximal barriers to get a bicritical graph with given property.
 Let $G_2=G_1/(\overline{V(Q_1)}\to \overline{q_1})$. If $G_2$ has a maximal nontrivial barrier $B_2$, similar to Claim B, we may assume that  $Q_2$ is a nontrivial component of $G_2-B_2$ such that $V(Q_2)\subset V(G)$, and $q_2\in W_{H(G_2,B_2)}$, where $q_2$ is obtained by contracting the odd component $Q_2$ of $G_2-B_2$. And then let $G_3=G_2/(\overline{V(Q_2)}\to \overline{q_2})$....
 Continue above steps, we finally  obtain a matching covered graph $G_s$ satisfying the following:\\
1) $G_s$ has a maximal nontrivial barrier $B_s$ and  $V(G_s)\setminus\{\overline{q_{s-1}}\}\subset V(G)$; and\\
2) there exists a nontrivial component $Q_s$ of $G_s-B_s$,  $G_s/(\overline{Q_s}\to\overline{q_s})$ has no nontrivial barriers, $V(Q_s)\subset V(G)$ and $q_s\in W_{H(G_s,B_s)}$, where $q_s$ is obtained by contracting the odd component $Q_s$ of $G_s-B_s$.\\

Let $G'=G_s/(\overline{Q_s}\to\overline{q_s})$.
{Note that $G'$ is a bicritical graph, $V(G')\setminus\{\overline{q_s}\}\subset V(G)$ and $G$ is a minimal matching covered graph.
If $G'$ has a removable edge, then this edge is incident with $\overline{q_s}$ by Lemma \ref{lem:re_also_re}.}
Noting $q_s\in W_{H(G_s,B_s)}$,  $q_s$ is incident with a removable edge in
$H(G_s,B_s)$, say $e$. Then the corresponding edge of $e$ (incident with $\overline{q_s}$) in $G'$ is not removable in $G'$ by Lemma \ref{lem:re_also_re} again.
By Lemma \ref{lem:WL-bicritical},
there exists a vertex $v\in V(G')\setminus\{\overline{q_s}\}$ such that $d_{G'}(v)=3$. As $V(G')\setminus\{\overline{q_s}\}\subset V(G)$, we have $d_{G}(v)=3$, contradicting the assumption that $\delta(G)\ge4$. So the theorem holds.
 \end{proof}

\end{document}